\documentclass[12pt]{amsart}

\usepackage{amsmath,amsfonts,amsthm,amsopn,cite,mathrsfs} 
\usepackage{epsfig,verbatim}
\usepackage{subfigure}
\usepackage{enumerate}
\usepackage[pdftex,%
plainpages=false,%
pdfpagelabels,%
hypertexnames=true,% true makes index work
colorlinks=true,%
linkcolor=black,%
citecolor=black,%
bookmarksopen=true,%
bookmarksopenlevel=3,%
bookmarksnumbered=true]{hyperref}

\allowdisplaybreaks

\newtheorem{thm}{Theorem}[section]
\newtheorem{prop}[thm]{Proposition}
\newtheorem{lem}[thm]{Lemma}
\newtheorem{cor}[thm]{Corollary}

\theoremstyle{definition}
\newtheorem{defi}[thm]{Definition}
\newtheorem{rmk}[thm]{Remark}

\newtheorem{qtn}{Question}
\renewcommand{\epsilon}{\varepsilon}
\renewcommand{\phi}{\varphi}

\makeatletter
\DeclareRobustCommand\widecheck[1]{{\mathpalette\@widecheck{#1}}}
\def\@widecheck#1#2{%
    \setbox\z@\hbox{\m@th$#1#2$}%
    \setbox\tw@\hbox{\m@th$#1%
       \widehat{%
          \vrule\@width\z@\@height\ht\z@
          \vrule\@height\z@\@width\wd\z@}$}%
    \dp\tw@-\ht\z@
    \@tempdima\ht\z@ \advance\@tempdima2\ht\tw@ \divide\@tempdima\thr@@
    \setbox\tw@\hbox{%
       \raise\@tempdima\hbox{\scalebox{1}[-1]{\lower\@tempdima\box
\tw@}}}%
    {\ooalign{\box\tw@ \cr \box\z@}}}
\makeatother

\newcommand{\vertiii}[1]{{\left\vert\kern-0.35ex\left\vert\kern-0.35ex\left\vert #1 \right\vert\kern-0.35ex\right\vert\kern-0.35ex\right\vert}}

\DeclareMathOperator{\co}{\mathrm{c}_0}

\DeclareMathOperator{\conv}{\mathrm{conv}}
\begin{document}
\begin{abstract}
For a rather general Banach space $X$, we prove that a nonempty closed convex bounded set $C\subset X$ is weakly compact if and only if every nonempty closed convex subset of $C$ has the fixed point property for the class of bi-Lipschitz affine maps. This theorem significantly complements and generalizes to some extent a known result of Benavides, Jap\'on-Pineda and Prus published in 2004. The proof is based on basic sequences techniques and involves clever constructions of fixed-point free affine maps under the lack of weak compactness. In fact, this result can be strengthened when $X$ fulfills Pe\l czy\'nski's property $(u)$.
\end{abstract}

\title[Weak compactness and FPP]{Weak compactness and fixed point property for affine bi-Lipschitz maps}
\author{Cleon S. Barroso}
\address{Departamento de Matem\'atica \\
Universidade Federal do Cear\'a \\
Campus do Pici \\
60455-360 Fortaleza, CE, Brazil}
\email{cleonbar@mat.ufc.br}
\author{Valdir Ferreira}
\address{Centro de Ci\^encias e Tecnologia \\
Universidade Federal do Cariri \\
 Bairro Cidade Universit\'aria\\
    63048-080 Juazeiro do Norte, CE, Brazil}
\email{valdir.ferreira@ufca.edu.br}

\subjclass[2000]{}
\date{}
% \thanks{K.\ G.\ was
%   supported in part by the  project P26273 - N25  of the
% Austrian Science Fund (FWF)}
\maketitle

\setcounter{tocdepth}{1}
\tableofcontents

\section{Introduction}\label{intro}

Describing and understanding topological phenomena remains one of the most active topics in functional analysis. The problem of describing weak compactness has so far particularly been a topic of great interest. In this paper we are concerned with the problem of whether compactness can be interpreted by the metric fixed point property (FPP). Recall that a topological space $C$ is said to have the FPP for a class $\mathcal{M}$ of maps if every $f\in \mathcal{M}$ with $f(C)\subset C$ has a fixed point. This problem has been studied from a number of topological viewpoints by several authors, see e.g. \cite{Kl,Flo,LS,DM,BKR,BPP} and references therein. In the purely metric context, it is often subjected to structural considerations. This can be seen in several works where weak compactness constitutes the FPP for nonexpansive ($1$-Lipschitz) affine  mappings. For example, Lennard and Nezir \cite{LN} proved that if a Banach space $X$ contains a basic sequence $(x_n)$ asymptotically isometric to the $\co$-summing basis, then its closed convex hull $\overline{\conv}\big( \{ x_n\}\big)$ fails the FPP for affine nonexpansive mappings. Typically, in theses cases, the set $\overline{\conv}\big( \{ x_n\}\big)$ is not weakly compact.

An interesting relaxation of the FPP is the generic-FPP ($\mathcal{G}$-FPP), a notion first proposed in \cite{BPP}.  For a convex subset $M$ of a topological vector space $X$, denote by $\mathcal{B}(M)$ the family of all nonempty bounded, closed convex subsets of $M$.

\begin{defi}[\cite{BPP}]\label{def:1sec1} A nonempty set $C\in\mathcal{B}(X)$ is said to have the $\mathcal{G}$-FPP for a class $\mathcal{M}$ of mappings if whenever $K\in \mathcal{B}(C)$ then every mapping $f\in \mathcal{M}$ with $f(K)\subset K$ has a fixed point.
\end{defi}

There is quite a lot known on $\mathcal{G}$-FPP. For instance, Dowling, Lennard and Turett \cite{DLT1,DLT2} proved that when $X$ is either $\co$, $L_1(0,1)$ or $\ell_1$ sets $C\in \mathcal{B}(X)$ are weakly compact if and only if they have the $\mathcal{G}$-FPP for affine nonexpansive maps. In  2004 Benavides, Jap\'on-Pineda and Prus proved, among other important results, the following facts.

\begin{thm}[(Benavides, Jap\'on Pineda and Prus \cite{BPP})]\label{thm:BPP} Let $X$ be a Banach space and $C\in \mathcal{B}(X)$. Then
\begin{enumerate}
\renewcommand{\labelenumi}{(\roman{enumi})}
\item $C$ is weakly compact if and only if $C$ has the $\mathcal{G}$-FPP for continuous affine maps. 
\item If $X$ is either $\co$ (equipped with the supremum norm $\| \cdot\|_\infty$) or  $J_p$ (the James space), then $C$ is weakly compact if and only if $C$ has the $\mathcal{G}$-FPP for uniformly Lipschitzian affine maps.
\item If $X$ is an $L$-embedded Banach space, then $C$ is weakly compact if and only if it has the $\mathcal{G}$-FPP for nonexpansive affine mappings.
\end{enumerate}
\end{thm}

Recall that a map $f\colon C\to X$ is said to be uniformly Lipschitz if
\[
\sup_{x\neq y\in C,\, p\in \mathbb{N}} \frac{ \| f^p(x)  - f^p(y)\|}{\| x - y\|}<\infty,
\]
where $f^p$ denotes the $p^{\textrm{th}}$ iteration of the mapping $f$. If, in addition, its inverse $f^{-1}$ is uniformly Lipschitz then $f$ is said to be uniformly bi-Lipschitz. Therefore, if $f$ is nonexpansive then it obviously is uniformly Lipschitz. It worths stressing that norm-continuous affine maps are in fact weakly continuous. Thus, one direction of the statements in Theorem \ref{thm:BPP} easily follows from Schauder-Tychonoff's fixed point theorem as pointed out in \cite{BPP}. 

At first sight one may be tempted to characterize weak-compactness in terms of the $\mathcal{G}$-FPP for nonexpansive maps. However this is not generally true. Indeed, in 2008 P.-K. Lin \cite{Lin} equipped $\ell_1$ with the equivalent norm
\[
\vertiii{ x }_{\mathscr{L}}=\sup_{k\in\mathbb{N}}\frac{8^k}{1+8^k}\sum_{n=k}^\infty | x(n)|\quad\textrm{for}\quad x=(x(n))_{n=1}^\infty \in \ell_1,
\]
and proved that every $C\in \mathcal{B}\big((\ell_1, \vertiii{\cdot}_{\mathscr{L}})\big)$ has the FPP for nonexpansive maps. Hence the unit ball $B_{(\ell_1, \vertiii{\cdot}_{\mathscr{L}})}$ has the $\mathcal{G}$-$FPP$ for affine nonexpansive maps, but fails to be weakly compact. 

Another interesting example is highlighted by the following result from the recent literature, due to T. Gallagher, C. Lennard and R. Popescu:

\begin{thm}[\cite{GLP}] Let $c$ be the Banach space of convergent scalar sequences. Then there exists a non-weakly compact set $C\in \mathcal{B}\big( (c, \|\cdot\|_\infty)\big)$ with the FPP for nonexpansive mappings. 
\end{thm}

It is natural to ask, therefore, whether weak compactness describes the $\mathcal{G}$-FPP for affine uniformly Lipschitz maps in arbitrary Banach spaces. To make this more precise, we formulate the following:

\begin{qtn}\label{qtn:1} Let $X$ be a Banach space and $C\in \mathcal{B}(X)$. Assume that $C$ is not weakly compact. Does there exist a set $K\in \mathcal{B}(C)$ and an affine uniformly Lipschitz map $f\colon K\to K$ that is fixed-point free? 
\end{qtn}

In a naive way, one could try to get a wide-$(s)$ sequence which uniformly dominates all of its subsequences; that is, a basic sequence $(x_n)$ such that for some positive constants $d$ and $D$ and every increasing sequence of integers $(n_i)\subset \mathbb{N}$, the following inequalities hold for all $n\in \mathbb{N}$ and all choice of scalars $(a_i)_{i=1}^n$
\begin{equation}\label{eqn:1int}
d\left| \sum_{i=1}^n a_i\right| \leq \left\| \sum_{i=1}^n a_i x_{n_i}\right\| \leq D \left\| \sum_{i=1}^n a_i x_i\right\|.
\end{equation}

This certainly obstructs the class of affine uniformly Lipschitz maps from having the $\mathcal{G}$-FPP. However this property has a strong unconditionality character. Indeed, subsymmetric or quasi-subsymmetric basis (in sense of \cite[Corollary 2.7]{ABDS}) are examples of basic sequences of this kind. So, it might not be so easy to get them since unconditional basic sequences may not exist at all \cite{GM}. 

Another possibility would be try to get wide-$(s)$ sequences $(x_n)$ that dominate their shift subsequences $(x_{n+p})$, but uniformly on $p$. Typically, this happens when special structures are available as, for example, those equivalent  to $\co$ or $\ell_1$ as well (cf. also \cite[Theorem 1]{DLT1}, \cite[Theorem 4.2]{BPP}, \cite{LN} and \cite[Proposition 2.5.14]{MN}). Such a possibility would however imply  that the shift operator induced by $(x_n)$ would be continuous. But this might be notoriously difficult, or even generally impossible. One of the reasons is that the class of Hereditarily Indecomposable spaces (spaces that have no decomposable subspaces, cf. \cite{GM}) do not admit shift-equivalent basic sequences, that is, sequences $(x_n)$ which are equivalent to its right-shift $(x_{n+1})$. Moreover, the Banach space $G$ was constructed by Gowers in \cite{G} has an unconditional basis for which the right shift operator is not norm-bounded. These facts seems to indicate that there is no hope to solving Question \ref{qtn:1} by considering shift like maps.

The first main result of this paper solves Question \ref{qtn:1} for the class of affine bi-Lipschitz maps. Precisely, it will be proved that if $X$ is a general Banach space and $C\in \mathcal{B}(X)$ is not weakly compact then it fails the $\mathcal{G}$-FPP for the class of bi-Lipschitz affine maps. Let us stress that the clever idea behind the proof is to build inside $C$ a basic sequence $(x_n)$ which dominates the summing basis of $\co$  and yet is equivalent to some of its non-trivial convex basis (see precise Definition \ref{def:3sec2}). This will give rise to a fixed-point free bi-Lipschitz affine map $f$ leaving invariant a set $K\in \mathcal{B}(C)$. As we shall see, the set $K$ is precisely the closed convex hull of $(x_n)$. As regards the map $f$, it will be essentially taken as the sum of a diagonal operator and a weighted shift map with properly chosen coefficients. This yields a new construction in metric fixed point theory and can make more transparent the challenges behind Question \ref{qtn:1}. To prove that $f$ is bi-Lipschitz we rely on a key lemma on affinely equivalent basic sequences. We also point out that our approach differs from that in \cite{BPP} where, because of the special nature of the spaces considered there, bilateral and right-shift maps were successfully used. The second main result is that one can affirmatively solve Question \ref{qtn:1} in spaces with the Pe\l czy\'nski's property $(u)$. The proof uses a local version of a classical result of James proved for spaces with unconditional basis. 

The remainder of the paper is organized as follows. In Section 2 we will set up the notation and terminology adopted in this work. In Section 3 we slightly recover a few ideas behind clever constructions of fixed-point free maps under the lack of weak compactness. In Section 4 contains a fundamental lemma concerning a notion of affinely equivalent sequences introduced by Pe\l czy\'nski and Singer. Section 6 contains a local version of a result of James which describes the internal structure of bounded, closed convex sets in spaces with property $(u)$. In Section 6 we formally state and prove the main result of this paper. Finally, in Section 7 we state and prove the second main result of this paper.

%%%%%%%%%%%%%%%%%
%\newpage

\section{Notation and basic terminology}\label{sec2:Notation}
Throughout this paper $X$ will denote a Banach space. The notation used here is standard. In particular, a sequence $(x_n)$ in $X$ is called a basic sequence if it is a Schauder basis for its closed linear span $[x_n]$. In this case $\mathcal{K}$ will stand for the basic constant of $(x_n)$. Further, we will also denote by $P_n$ and $R_n$ the natural basis projections given by 
\[
P_n x= \sum_{i=1}^n x^*_i(x) x_i\quad \textrm{ and } \quad R_nx= x - P_nx,\quad x\in [x_n]
\] 
where $\{ x^*_i\}_{i=1}^\infty$ are the biorthogonal functionals associated with $(x_n)$. Recall that $\mathcal{K}:=\sup_n\| P_n\|$. By $\mathrm{c}_{00}$ we denote the vector space of sequences of real numbers which eventually vanish. Let us now recall a few well-known notions from Banach space theory.

\begin{defi}\label{def:2sec2} Let $(x_n)\subset X$ and $(y_n)\subset Y$ be two sequences, where $X, Y$ are Banach spaces. The sequence $(x_n)$ is said to dominate the sequence $(y_n)$ if there exists a constant $L>0$ so that 
\[
\Big\| \sum_{n=1}^\infty a_n y_n \Big\| \leq L \Big\| \sum_{n=1}^\infty a_n x_n \Big\|, 
\]
for all sequence $(a_n)\in \mathrm{c}_{00}$.
\end{defi}

Observe that when $(x_n)$ and $(y_n)$ are both basic sequences, to say that $(x_n)$  dominates $(y_n)$ is the same as to say that the map $x_n\mapsto y_n$ extends to a linear bounded map between $[x_n]$ and $[y_n]$. The sequences $(x_n)$ and $(y_n)$ are said to be equivalent (also called $L$-equivalent, with $L\geq 1$) and one writes $(x_n)\sim_L (y_n)$, if for any $(a_i)\in \mathrm{c}_{00}$ one has that
\[
\frac{1}{L} \Big\| \sum_{i=1}^\infty a_i x_i \Big\| \leq \Big\| \sum_{i=1}^\infty a_i y_i \Big\| \leq L \Big\| \sum_{i=1}^\infty a_i x_i \Big\|. 
\]
The {\it summing basis} of $\co$ is the sequence $(\chi_{\{ 1,2, \dots, n\}})$ in $\co$ where for $n\in \mathbb{N}$, $\chi_{\{ 1,2, \dots, n\}}$ is defined by
\[
\chi_{\{1,2, \dots, n\}}=e_1 + e_2 + \dots + e_n,
\]
and $(e_n)$ being the canonical basis of $\co$. It is well known that the sequence $(\chi_{\{1, 2,\dots, n\}})_n$ defines a Schauder basis for $( \co, \| \cdot\|_\infty)$. A sequence $(x_n)$ in a Banach space $X$ is then said to be equivalent to the summing basis of $\co$ if 
\[
(x_n)\sim_L (\chi_{\{1, 2,\dots, n\}})\,\, \textrm{ for some } L\geq 1.
\]

\begin{defi} A sequence $(x_n)$ in $X$ is called  {\it semi-normalized} if 
\[
0<\inf_n\| x_n\|\leq \sup_n\| x_n\|<\infty. 
\] 
\end{defi}

The following additional notions were introduced by H. Rosenthal.

\begin{defi}[(Rosenhtal \cite{Ro})] A seminormalized sequence $(x_n)$ in $X$ is called:
\begin{enumerate}
\item A non-trivial weak Cauchy sequence if it is weak Cauchy and non-weakly convergent. 
\item A wide-$(s)$ sequence if $(x_n)$ is basic and dominates the summing basis of $\co$.
\item An $(s)$-sequence if $(x_n)$ is weak-Cauchy and a wide-$(s)$ sequence.
\item {\it Strongly summing} if it is a weak-Cauchy basic sequence so that whenever $(a_i)$ is a sequence of scalars with $\sup_n\big\| \sum_{i=1}^n a_i x_i\big\| <\infty$, $\sum_i a_i$ converges. 
\end{enumerate}
\end{defi}

Rosenthal's $\mathrm{c}_0$-theorem \cite{Ro} ensures that every non-trivial weak-Cauchy sequence in $X$ has either a strongly summing subsequence or a convex block basis which is equivalent to the summing basis of $\mathrm{c}_0$. Finally, recall that a sequence of non-zero elements  $(z_n)$ of $X$ is called a convex block basis of a given sequence $(x_n)\subset X$ if there exist integers $n_1<n_2<\dots$ and scalars $c_1, c_2,\dots$ so that
\begin{enumerate}
\item[(i)] $c_i\geq 0$ for all $i$ and $\sum_{i=n_j+1}^{n_{j+1}} c_i=1$ for all $j$.\vskip .1cm
\item[(ii)] $z_j=\sum_{i=n_j + 1}^{n_{j+1}}c_i x_i$ for all $j$.
\end{enumerate}

%%%%%%%%%%%%%%%%%%%%%%%%%%%%%%%%%%%%%%%%%%  SECTION 3

\section{Convex basic sequences}\label{sec:3}
The construction of affine fixed-point free maps usually relies on maps  which are defined by taking suitable convex combinations of some basic sequence $(x_n)$ in $X$. For example, in  \cite{BPP} the following maps were considered in the proof of Theorem \ref{thm:BPP}:
\[
f_0\Big( \sum_{n=1}^\infty t_n x_n\Big) = \sum_{n=1}^\infty t_n x_{n+1},
\]
and
\[
f_1\Big( \sum_{n=1}^\infty t_n x_n \Big) = t_2 x_1 + \sum_{n=1}^\infty t_{ 2n-1} x_{2n+1}  + \sum_{n=2}^\infty t_{2n} x_{2n-2}.
\]
It is interesting to mention that, according to the terminology of \cite{BPP}, $f_0$ and $f_1$ are respectively a unilateral shift and a bilateral shift map. 

As another instance, the authors in \cite{DLT2} have described weak compactness in $\co$ in terms of the $\mathcal{G}$-FPP for nonexpansive maps by considering the map:
\[
f_2\Big( \sum_{n=1}^\infty t_n x_n\Big) = \sum_{n\in \mathbb{N}} \sum_{j\in \mathbb{N}} \frac{1}{2^j} t_n x_{j+n}.
\]

Therefore if $X$ is structurally well-behaved these convex combinations can be dominated by $(x_n)$. This naturally reflects on the metric fixed point property. Thus, it seems reasonable to understand the structure of such combinations. As a step towards this direction, we consider the following slightly generalized notion of convex block sequences. 

\begin{defi}\label{def:3sec2} Let $(x_n)$ be a sequence in $X$. A sequence $(z_n)$ is called a convex basis of $(x_n)$ if $(z_n)$ is basic and for each $n\in \mathbb{N}$ there exist scalars $\{ \lambda^{(n)}_k\}_{k=1}^\infty$ in $[0,1]$ so that $\sum_{i=1}^\infty \lambda^{(n)}_i=1$ and $z_n=\sum_{i=1}^\infty \lambda^{(n)}_ix_i$.
\end{defi}

\begin{rmk}\label{rem:star} It is clear that every subsequence of a basic sequence $(x_n)$ is itself a convex basis of $(x_n)$. These subsequences will be referred here as trivial convex basis. 
\end{rmk}

As we have mentioned in the introduction, one may try to describe weak-compactness in terms of the $\mathcal{G}$-$FPP$ for the class of uniformly Lipschitz maps by trying to get wide-$(s)$ sequences satisfying (\ref{eqn:1int}). The proposition below shows however that spaces with the scalar-plus-compact property are not optimal environments for doing that.

\begin{prop} Let $(x_n)$ be a wide-$(s)$ sequence in $X$. Assume that $(z_n)$ is a convex basis of $(x_n)$ whose subsequences are dominated by $(x_n)$. Then $\mathcal{L}( [x_n])$ is non-separable.
\end{prop}

\begin{proof} We proceed as in \cite{ABDS} by obtaining uncountable many pairwise separated bounded linear operators on the space $[x_n]$. For each increasing sequence $(\kappa_n)$ in $\mathbb{N}$ define $T_{(\kappa_n)}\colon [x_n]\to [x_n]$ by $T_{(\kappa_n)}(x)= \sum_{n=1}^\infty x^*_n(x) z_{\kappa_n}$. By assumption each map $T_{(\kappa_n)}\in \mathcal{L}([x_n])$. Moreover, if  $(\kappa_n)$ and $(\ell_n)$ are two different increasing sequences in $\mathbb{N}$ then for some $j\in \mathbb{N}$ so that $\kappa_j \neq \ell_j$ we have
\begin{eqnarray*}
\| T_{(\kappa_n)} - T_{(\ell_n)}\| &\geq& \left\| \sum_{n=1}^\infty \left( x^*_n\bigg(\frac{x_j}{\| x_j\|}\bigg)z_{\kappa_n} - x^*_n\bigg( \frac{ x_j}{\| x_j\| } \bigg) z_{\ell_n}\right)\right\|\\
&=& \frac{1}{\| x_j\|} \| z_{\kappa_j} - z_{\ell_j}\|\geq \frac{\inf_n\| z_n\|}{\mathcal{K}\sup_n\| x_n\|}>0,
\end{eqnarray*}
where $\mathcal{K}$ denotes the basic constant of $(z_n)$. The penultimate inequality above follows easily from the fact that $(x_n)$ is $\mathcal{K}$--basic, while the last one is a direct consequence of $(x_n)$ being wide-$(s)$  which in turn implies $\inf_n\| z_n\|>0$. 
\end{proof}

Our first main result relies on the selection of non-trivial convex bases that must be structurally well behaved. This involves two important steps. The first one concerns the selection of wide-$(s)$ subsequences. To this end we will rely on the following result of Rosenthal (\cite[Proposition2]{Ro1}), the proof of which will be included here for reader's convenience. The second one concerns clever constructions of convex bases of wide-$(s)$ sequences, this precisely being the content of the next sections.

\begin{prop}\label{prop:selection} Let $X$ be a Banach space and $(y_n)$ be a seminormalized sequence in $X$. Assume that no subsequence of $(y_n)$ is weakly convergent. Then $(y_n)$ admits a wide-$(s)$ subsequence. 
\end{prop}

\begin{proof}
If $(y_n)$ has no weak-Cauchy subsequence, then $(y_n)$ has an $\ell_1$-subsequence $(x_n)$ by the Rosenthal $\ell_1$-theorem. It is easy to see in this case that $(x_n)$ is wide-$(s)$. If otherwise $(y_n)$ has a weak-Cauchy subsequence $(y_{n_k})$, then from our assumption and \cite[p. 707]{Ro} we get that $(y_{n_k})$ is a non-trivial weak-Cauchy sequence. By \cite[Proposition 2.2]{Ro}, $(y_{n_k})$ has an $(s)$-subsequence $(x_n)$. This shows in particular that $(x_n)$ is wide-$(s)$ and concludes the proof. 
\end{proof}

%%%%%%%%%%%%%%%%%%%%%%%%%%%%%%%%%%%%

\section{A lemma on affinely equivalent basic sequences}\label{sec:4}
In this section we will establish a key lemma crucial for the proof of our first main result. It concerns the following notion introduced by Pe\l czy\'nski and Singer \cite{PS}.

\begin{defi}[Pe\l czy\'nski--Singer] A basic sequence $(x_n)$ in a Banach space $X$ is said to be affinely equivalent to a sequence $(y_n)$ if there exists a sequence of scalars $\alpha_n \neq 0$ such that $(x_n)$ and $(\alpha_n y_n)$ are equivalent. 
\end{defi}

The proof of our key lemma is based on the following result of H\'ajek and Johanis (\cite[Lemma 5-(a)]{HJ}). For completeness we will provide a more direct proof. 

\begin{prop}[H\'ajek--Johanis]\label{lem:HJ} Let $X$ be a Banach space with a Schauder basis $\{x_n, x^*_n\}_{n=1}^\infty$. Assume that $\| R_n\| = 1$ for each $n\in \mathbb{N}$ and $\{ \alpha_n\}$ is non-decreasing real sequence in $(0,1]$. Then
\[
\Big\| \sum_{n=1}^\infty \alpha_n x^*_n(x) x_n\Big\| \leq \Big\| \sum_{n=1}^\infty x^*_n(x) x_n\Big\|\quad \textrm{ for each } x\in X. 
\]
\end{prop}

\begin{proof}
Fix $x\in X$. For each $N>1$, we define a new sequence $(y_{n,N})_{n=1}^\infty$ in $X$ by putting, for $n\geq N$, $y_{n, N}=\alpha_N x$ and 
\[
y_{n, N} =R_n y_{n+1, N} + \frac{\alpha_n}{\alpha_{n+1}}P_n y_{n+1, N}\quad\textrm{for } 1\leq n< N. 
\]
As in \cite{HJ} a direct computation shows $\| y_{n,N}\| \leq \| y_{n+1, N}\|$ for $n<N$.  Moreover, an easy induction argument implies 
\[
y_{n, N} = \alpha_n P_{n-1} x + \sum_{k=n}^N \alpha_k x^*_k(x) x_k + \alpha_N R_N x,\quad 1\leq n< N.
\]
Thus $y_{1,N} =\sum_{k=1}^N \alpha_k x^*_k(x) x_k + \alpha_N R_Nx$ and hence
\[
\Bigg\|\sum_{k=1}^N \alpha_k x^*_k(x) x_k + \alpha_N R_Nx\Bigg\| \leq \alpha_N \| x\|\leq \| x\| 
\]
Now it is easy to show, using that $\{\alpha_n\}$ is non-decreasing and that $(x_i)$ is basic, that the series $\sum_k \alpha_k e^*_k(x)e_k$ converges in $X$. Notice further that $\alpha_N R_Nx\to 0$. So, the result follows by taking the limit as $N\to\infty$. 

\end{proof}

%%%%%%
We are now ready to state and prove the main result of this section which yields a sufficient condition for a basic sequence to be affinely equivalent to itself. 

\begin{lem}[Key Lemma]\label{prop:1HJ} Let $X$ be a Banach space and $(x_n)$ a basic sequence in $X$. Assume that $\{\alpha_n\}\subset (0,1]$ is a non-decreasing sequence of real numbers. Then $(x_n)\sim_ {2\mathcal{K}/\alpha_1} (\alpha_n x_n)$ where $\mathcal{K}$ is the basic constant of $(x_n)$.
\end{lem}

\begin{proof} Let $L=2\mathcal{K}/\alpha_1$. The fact that $(x_n)$ $L$-dominates $(\alpha_n x_n)$ follows directly from Lemma \ref{lem:HJ}. To see this, it suffices to take an equivalent norm $\vertiii{\cdot}$ on $[x_n]$ so that in the new norm the basis $(x_n)$ fulfills $\vertiii{R_n}\leq 1$. Indeed, denote by $P_I$ the natural projection over a finite interval $I\subset\mathbb{N}$ and define a new norm on $[x_n]$ by
\[
\vertiii{x} = \sup\Big\{ \| P_I x\|\,\colon \, I\subset\mathbb{N},\, I \textrm{ finite interval}\Big\} \quad \textrm{for } x\in [x_n].
\]
Hence $\|\cdot\|$ and $\vertiii{\cdot}$ are equivalent norms on $[x_n]$ with 
\[
\max\{ \vertiii{P_n}, \vertiii{R_n}\}\leq 1,\quad \textrm{ for all } n\in\mathbb{N}. 
\]
On the other hand, as $R^2_n= R_n$ implies $\| R_n\|\geq 1$, we get that $\vertiii{R_n}=1$ for all $n\in \mathbb{N}$.  Moreover, observe that
\[
\| x\| \leq \vertiii{x} \leq 2\mathcal{K}\| x\|\quad \textrm{ for each } x\in [x_n].
\]
Thus this combined with Lemma \ref{lem:HJ} implies that, for every $(a_i)\in c_{00}$
\begin{eqnarray*}
\Big\| \sum_{i=1}^\infty a_i \alpha_i x_i\Big\| \leq \vertiii{ \sum_{i=1}^\infty a_i \alpha_i x_i}&\leq& \vertiii{\sum_{i=1}^\infty a_i x_i}\\
&\leq& 2\mathcal{K} \Big\| \sum_{i=1}^\infty a_i x_i \Big\|\leq \frac{2\mathcal{K} }{\alpha_1}\Big\| \sum_{i=1}^\infty a_i x_i \Big\|.
\end{eqnarray*}
To prove the reverse inequality, fix $N\in \mathbb{N}$ and pick any sequence of scalars $(a_i)_{i=1}^N$. Now combining the Abel's partial summation
\[
\sum_{n=1}^N a_n\alpha_n x_n =\sum_{n=1}^{N-1} (\alpha_n - \alpha_{n+1}) \sum_{i=1}^n a_i x_i  + \alpha_N \sum_{i=1}^N a_i x_i,
\]
with the $\vertiii{\cdot}$-monotonicity of $(x_n)$ (i.e., $\vertiii{P_n}\leq 1$ for any $n$), it follows that 
\begin{eqnarray*}
\vertiii{ \sum_{n=1}^N a_n \alpha_n x_n }&\geq& \alpha_N \vertiii{ \sum_{i=1}^N a_i x_i } - \sum_{n=1}^{N-1} ( \alpha_{n+1} - \alpha_n) \vertiii{ \sum_{i=1}^n a_i x_i}\\
&\geq& \alpha_1 \vertiii{ \sum_{i=1}^N a_i x_i }
\end{eqnarray*}
which in turn yields
\[
2\mathcal{K}\left\| \sum_{n=1}^N a_n \alpha_n x_n \right\| \geq \alpha_1 \left\| \sum_{n=1}^N a_n x_n \right\|.
\]
The proof is complete. 
\end{proof}

%%%%%%%%%%%

\section{Bounded, closed convex sets in spaces with property $(u)$}
Recognizing local structures in Banach spaces are relevant in the study of the metric fixed point theory. The main result of this section supplies a local version of a well-known result of James. It is concerned with the internal structure of bounded, closed convex sets in spaces with Pe\l czy\'nski's property $(u)$. 

\begin{defi}[Pe\l czy\'nski] An infinite dimensional Banach space $X$ is said to have property $(u)$ if for every weak Cauchy sequence $(y_n)$ in $X$, there exists a sequence $(x_n)\subset X$ satisfying the properties below:
\begin{enumerate}
\item $\sum_{n=1}^\infty x_n$ is weakly unconditionally Cauchy $(WUC)$ series, i.e
\[
\sum_{n=1}^\infty | x^* ( x_n) | <\infty \quad\textrm{ for all } x^*\in X^*.
\]
\item $(y_n - \sum_{i=1}^n x_i)_n$ converges weakly to zero.
\end{enumerate}
\end{defi}

\begin{rmk} A few known facts are in order: Banach spaces with an unconditional basis have property $(u)$ (cf. \cite[Proposition 3.5.4]{AK}). Other examples of spaces satisfying the property $(u)$ can be found in \cite{GL} where, for instance, it is shown that $L$-embedded spaces enjoy this property.  The classical James' space $\mathcal{J}_2$ is an example of a space which fails property $(u)$. 
\end{rmk}

\begin{lem}\label{lem:KL2} Let $X$ be a Banach space with the property $(u)$ and $C\in\mathcal{B}(X)$. Then either $C$ is weakly compact, $C$ contains an $\ell_1$-sequence or $C$ contains a $\co$-summing basic sequence. 
\end{lem}

\begin{proof} Suppose $C$ is weakly compact. By \cite[Proposition 2-(a)]{Ro}, $C$ cannot contain wide-$(s)$ sequences. So, it does not contain neither $\ell_1$-basic sequences nor $\co$-summing basic sequences, as well. Assume that $C$ is not weakly compact. Then it contains either an $\ell_1$-sequence or not. If so, the result follows. Otherwise, $C$ must contain a $\co$-summing basic sequence. Indeed, let $(y_n)\subset C$ be a weak-Cauchy sequence without weak convergent subsequences. This is possible thanks to Eberlein-\v{S}mulian's theorem and as well as Rosenthal's $\ell_1$-theorem. If $X$ has the property $(u)$, then so does the space $[(y_n)_n]$ (see \cite{Pel} (cf. also \cite[Proposition 3.5.4]{AK}). Therefore, by a result of of Haydon, Odell and Rosenthal \cite{HOR} (cf. also \cite[p. 154]{KO}), $(y_n)$ has a convex block basis $(x_n)$ which is equivalent to the summing basis of $\co$. This concludes the proof. 
\end{proof}

\begin{rmk} It is worth to mention that if $X$ has an unconditional basis then an even more strong result can be stated: 
\begin{lem} Let $X$ be a Banach space and $C\in \mathcal{B}(X)$. Assume that $X$ has an unconditional basis. Then exclusively either $C$ is weakly compact, $C$ contains an $\ell_1$-sequence or $C$ contains a $\co$-summing basic sequence. 
\end{lem}

\begin{proof} In view of the previous result it suffices to prove the result assuming that $C$ is not weakly compact. If $C$ contains an $\ell_1$-basic sequence, so does $X$. Since $X$ has unconditional basis, by James' Theorem \cite{J1} $X$ does not contain any isomorphic copy of $\co$. Hence $C$ contains no $\co$-summing basic sequences. Suppose now that $C$ contains no $\ell_1$-basic sequences. As before, we claim that $C$ contains a $\co$-summing basic sequence. The proof of this assertion follows the same steps in the final part of the proof of Lemma \ref{lem:KL2}. %Indeed, let $(y_n)\subset C$ be a weak-Cauchy sequence with no weakly convergent subsequences. Notice that this is possible thanks to Rosenthal's $\ell_1$-Theorem. Since $X$ has unconditional basis the space $[(y_n)_n]$ fulfills Pe\l czy\'nski's property $(u)$, see \cite{Pel} (cf. also \cite[Proposition 3.5.4]{AK}). Therefore, by a result of of Haydon, Odell and Rosenthal \cite{HOR} (cf. also \cite[p. 154]{KO}), $(y_n)$ has a convex block basis $(x_n)$ which is equivalent to the summing basis of $\co$. This concludes the proof. 
\end{proof}

\end{rmk}

%%%%%%%%%%%%%%%%%%%%%%%%%%%%%%%%%%%%%%%

\section{The $\mathcal{G}$-FPP in arbitrary Banach spaces} 
Our first main result reads as follows. 

\begin{thm} Let $X$ be a Banach space and $C\in \mathcal{B}(X)$. Then $C$ is weakly compact if and only if $C$ has the $\mathcal{G}$-FPP for affine bi-Lipschitz maps. 
\end{thm}

\begin{proof}
As we have mentioned before, if $C$ is weakly compact then it has the $\mathcal{G}$-FPP for any class of norm-continuous affine maps. Thus only the converse direction needs to be proved. Assume then that $C$ is not weakly compact. By Eberlein-\v{S}mulian's Theorem, we can find a sequence $(y_n)$ in $C$ with no weakly convergent subsequences. Let $(x_n)$ be the wide-$(s)$ subsequence of $(y_n)$ given by Proposition \ref{prop:selection}. In order to prove the failure of the $\mathcal{G}$-FPP we need to exhibit a set $K\in \mathcal{B}(C)$ and a fixed-point free bi-Lipschitz affine map $f\colon K\to K$. As regards the set $K$, we let $K=\overline{\conv}(\{ x_n\})$.  Before starting the construction of $f$, we need to set up an useful formula for $K$. We claim:\vskip .1cm

\paragraph{\bf Claim:} $K=\big\{ \sum_{n=1}^\infty t_n x_n  \,\colon\, \textrm{each } t_n\geq 0\, \textrm{ and }\, \sum_{n=1}^\infty t_n=1\big\}$. 

\begin{proof}[Proof of Claim] Let 
\[
M= \Big\{ \sum_{n=1}^\infty t_n x_n  \,\colon\, \textrm{ each } t_n\geq 0 \textrm{ and } \sum_{n=1}^\infty t_n=1\Big\}. 
\]
First note that $M$ is closed in $C$. Indeed, assume that $\{u_k\}_{k=1}^\infty \subset M$ converges to $u\in C$. For each $k\in \mathbb{N}$, write $u_k=\sum_{n=1}^\infty t^{(k)}_n x_n$ where each $t^{(k)}_n \geq 0$ and $\sum_{n=1}^\infty t^{(k)}_n=1$. As $(x_n)$ is basic and $u\in [x_n]$ we may write $u= \sum_{n=1}^\infty t_n x_n$. It follows that $t_n= \lim_{k\to \infty} t^{(k)}_n\geq 0$. 

Now since $(x_n)$ is wide-$(s)$ there is a constant $L>0$ such that 
\begin{equation}\label{eqn:wide}
L \Big| \sum_{n=1}^\infty a_n\Big| \leq \Big\| \sum_{n=1}^\infty a_n x_n \Big\|\quad\forall (a_n)\in \ell_1. 
\end{equation}
Thus the series $\sum_{n=1}^\infty t_n$ converges and hence
\[
L\Big| 1 - \sum_{n=1}^\infty t_n\Big| \leq \Big\| \sum_{n=1}^\infty \big( t^{(k)}_n - t_n \big) x_n\Big\|= \| u_k - u \|\to 0,
\]
which implies $\sum_{n=1}^\infty t_n=1$ and so $u\in M$. Now since $M$ is closed convex and contains $(x_n)$ we obtain $K\subset M$. To prove the converse inclusion, let $u=\sum_{n=1}^\infty t_n x_n\in M$. We have to prove that $u\in K$.  Let $v\in K$ be fixed and define for $k\in \mathbb{N}$, 
\[
u_k =\big(1 - \sum_{n=1}^k t_n\big)\cdot v + \sum_{n=1}^k t_n x_n. 
\]
Then an easy computation shows we can conclude that $u_k \in K$ for all $k$ and, moreover, since $\sum_{n=1}^k t_n\to 1$ as $k\to\infty$, 
\[
\| u_k - u\| \leq \Big( 1 - \sum_{n=1}^k t_n \Big) \| v\| + \sup_{n\in \mathbb{N}}\| x_n\| \sum_{n=k+1}^\infty t_n\to 0,
\]
which implies $u\in K$, as it is closed. This proves the claim.
\end{proof}

With the set $K$ in hand, we proceed to construct the map $f$. Choose a sequence of scalars $(\alpha_n)$ satisfying the conditions:
\begin{enumerate}
\item $0<\alpha_n <1/2$ for $n\in \mathbb{N}$.
\item $\alpha_n \searrow 0$.
\item $\displaystyle\sum_{n=1}^\infty \alpha_n< \frac{1}{4\mathcal{K}} \frac{\inf_n \| x_n\|}{ \sup_n \| x_n\|}$.\vskip .2cm
\end{enumerate}
It is obvious that such numbers can be found. We then define $f\colon K \to K$ as follows: if $\sum_{n=1}^\infty t_n x_n\in K$, then
\[
f\Big( \sum_{n=1}^\infty t_n x_n \Big) = (1 - \alpha_1) t_1 x_1 + \sum_{n=2}^\infty \big( (1- \alpha_n)t_n + \alpha_{n-1}t_{n-1}\big) x_n.
\]
Clearly $f$ is an affine fixed point free self map of $K$. It remains to show that $f$ is bi-Lipschitz. In order to verify this, we let 
\[
z_n= (1 - \alpha_n) x_n  + \alpha_n x_{n+1},\quad n\in \mathbb{N}.
\]
Notice that $(z_n)$ is a non-trivial convex basis of $(x_n)$. Furthermore, 
\[
f(x) = \sum_{n=1}^\infty t_n z_n\quad \forall\,  x:=\sum_{n=1}^\infty t_n x_n \in K. 
\]
So, it is enough to prove the following:
\paragraph{Claim:} $(z_n)$ is equivalent to $(x_n)$.

\begin{proof}[Proof of Claim] First we note that $\inf_n \| z_n\|>0$. To see this, it suffices to note that if $(x_n)$ is wide-$(s)$ then for some constant $L>0$ such that (\ref{eqn:wide}) holds, we have 
\[
\| z_n\| = \| (1 - \alpha_n) x_n  + \alpha_n x_{n+1}\| \geq L \big| (1 - \alpha_n)  + \alpha_n \big|= L\quad\forall \,n\in \mathbb{N}.
\]
Thus $(z_n)$ is seminormalized. We shall show now that $(z_n)$ is basic. If, for $n\in \mathbb{N}$, we define
\[
w_n= (1 - \alpha_n) x_n
\]
then $(w_n)$ is equivalent to $(x_n)$, by Lemma \ref{prop:1HJ}. So, it is basic. Furthermore, $(z_n)$ is equivalent to $(w_n)$ by the Principle of Small Perturbations (\cite[Theorem 1.3.9]{AK}). For, it can be readily verified that
\[
\sum_{n=1}^\infty \frac{\| w_n - z_n \|}{\| w_n\|}< \frac{1}{2\mathcal{K}}.
\]
Thus $(z_n)$ is basic and is equivalent to $(x_n)$. This establishes the claim and completes the proof of theorem.
\end{proof}
\end{proof}

%%%%%%%%%%%%%%%%%

\section{The $\mathcal{G}$-FPP in spaces with Pe\l czy\'nski's property $(u)$}
In this section we give an affirmative answer for Question \ref{qtn:1} in spaces with the property $(u)$. More precisely, we obtain the following result. 

\begin{thm}\label{thm:6.1} Let $X$ be a Banach space with the property $(u)$. Then $C\in \mathcal{B}(X)$ is weakly compact if and only if it has the $\mathcal{G}$-FPP for the class of affine uniformly bi-Lipschitz maps. 
\end{thm}

\begin{proof} It suffices to prove the converse implication. Assume that $C$ is not weakly compact. By Lemma \ref{lem:KL2} either $C$ contains a $\ell_1$-basic sequence or it contains a $\co$-summing basic sequence. In either case we see that $C$ contains a wide-$(s)$ sequence $(x_n)$ so that $(x_{n+p})$ is equivalent to $(x_n)$, but uniformly on $p\in \mathbb{N}$. Hence for $K=\overline{\conv}\big( \{ x_n\}\big)$, the map $f\colon K\to K$  given by 
\[
f( x) = \sum_{i=1}^\infty t_n x_{n+1}\quad \textrm{for } \,x= \sum_{i=1}^\infty t_n x_n\in K,
\]
is affine, fixed point free and uniformly Lipschitz. This concludes the proof of the theorem. 
\end{proof}

An immediate corollary of Theorem \ref{thm:6.1} is

\begin{cor} Let $X$ be a Banach space. Assume that $X$ is either $L$-embedded or has the hereditary Dunford-Pettis property. Then $C\in \mathcal{B}(X)$ is weakly compact if and only if it has the $\mathcal{G}$-FPP for the class of affine uniformly bi-Lipschitz maps. 
\end{cor}

\begin{rmk} Recall \cite{D} that a Banach space $X$ is said to have the Dunford-Pettis property if for every pair of weakly null sequences $(x_n)\subset X$ and $(x^*_n)\subset X^*$ one has $\lim_{n\to \infty} \langle x_n , x^*_n\rangle=0$. Further, $X$ is said to have the hereditarily Dunford-Pettis if all of its closed subspaces have the Dunford-Pettis property. It is also known (cf. proof of \cite[Theorem 2.1]{KO}) that spaces with the hereditary Dunford-Pettis property have property $(u)$. 
\end{rmk}

%%%%%%%%%%%%%%%%%%%%%%%%%%%%%%%%%%%

 \nocite{*}
\bibliographystyle{cdraifplain}
\bibliography{xampl}

\end{document}